\documentclass[runningheads]{llncs}
\usepackage{marvosym}
\usepackage{soul}
\usepackage{comment}
\usepackage{ifsym}
\usepackage{hyperref}
\usepackage{orcidlink}
\usepackage{amssymb}
\usepackage{amsmath}
\usepackage{graphicx}
\usepackage{stmaryrd}
\usepackage{bussproofs}
\hypersetup{
    colorlinks=true,
    linkcolor=blue,
    }
\usepackage{tikz}
\usetikzlibrary{shapes.geometric, arrows, positioning}
\tikzstyle{box} = [rectangle, rounded corners, minimum width=2cm, minimum height=1cm,text centered, draw=white, fill=white]
\tikzstyle{arrow} = [->,>=stealth]

\begin{document}

\newcommand{\mrm}{\mathrm}
\newcommand{\mc}{\mathcal}
\newcommand{\tit}{\textit}
\newcommand{\mbf}{\mathbf}
\newcommand{\mbb}{\mathbb}
\newcommand{\msf}{\mathsf}
\newcommand{\mcr}{\mathscr}
\newcommand{\mfk}{\mathfrak}

\newcommand{\D}{\Diamond}
\newcommand{\B}{\Box}
\newcommand{\bd}{\blacklozenge}
\newcommand{\bb}{\blacksquare}
\newcommand{\cc}{\circ}
\newcommand{\bu}{\bullet}
\newcommand{\ns}{{\sim}}
\newcommand{\vd}{\vdash}
\newcommand{\hra}{\hookrightarrow}
\newcommand{\imp}{\rightarrow}
\newcommand{\pmi}{\leftarrow}
\newcommand{\bimp}{\leftrightarrow}
\newcommand{\Bimp}{\Leftrightarrow}
\newcommand{\Imp}{\Rightarrow}
\newcommand{\seq}{\Rightarrow}
\newcommand{\ua}{\uparrow}
\newcommand{\dr}{\downarrow}
\newcommand{\md}{\models}
\newcommand{\app}{\approx}
\newcommand{\ol}{\overline}
\newcommand{\mdp}{\models ^+}
\newcommand{\mdn}{\models ^-}

\newcommand{\den}[1]{\llbracket{#1}\rrbracket}
\newcommand{\tup}[1]{\langle{#1}\rangle}

\newcommand{\al}{\alpha}
\newcommand{\be}{\beta}
\newcommand{\ga}{\gamma}
\newcommand{\de}{\delta}
\newcommand{\si}{\sigma}
\newcommand{\Si}{\Sigma}
\newcommand{\Ga}{\Gamma}
\newcommand{\De}{\Delta}
\newcommand{\Th}{\Theta}
\newcommand{\ep}{\epsilon}
\newcommand{\ka}{\kappa}
\newcommand{\Ka}{\Kappa}
\newcommand{\la}{\lambda}
\newcommand{\La}{\Lambda}

\newcommand{\ntg}{\varnothing}
\newcommand{\union}{\cup}
\newcommand{\Union}{\bigcup}
\newcommand{\inter}{\cap}
\newcommand{\Inter}{\bigcap}
\newcommand{\sub}{\subseteq}

\newtheorem{fact}[definition]{Fact}

\title{On the Cut Elimination of Weak Intuitionistic Tense Logic\thanks{TBA}}
%
%
\author{Yiheng Wang\inst{1}\orcidlink{0000-0002-5975-2333} \and Yu Peng\inst{1}\orcidlink{0000-0001-6371-9862} \and Zhe Lin(\Letter)\orcidlink{0009-0000-0575-4942}\inst{2}}
\authorrunning{Y Wang and Z Lin}
%
\institute{Department of Philosophy, Institute of Logic and Cognition, Sun Yat-sen University, Guangzhou, China\\ \email{ianwang747@gmail.com}, \email{amberlogos@163.com} \and Department of Philosophy, Xiamen University, Xiamen, China\\ \email{pennyshaq@163.com}}
\maketitle  
\begin{abstract}
In this paper, we use a new method to prove cut-elimination of a weak variant of intuitionistic tense logic. This method focuses on splitting the contraction rule and cut rules. Further general theories and applications of this method shall be developed in the future.

\keywords{weak intuitionistic tense logic  \and cut-elimination \and tense logic}
\end{abstract}

\section{Introduction}\label{section1}

    Intuitionistic tense logic was introduced by Ewald in his work \cite{EW1986} in 1986. Within this framework, Ewald considered four tense operators, comprising two pairs of adjoint modalities: $(F (\D), H (\bb))$ and $(P (\bd),G (\B))$, which are integrated within intuitionistic logic. This foundational logic, denoted as $\msf{IK}_t$, was further explored by Figallo and Pelaitay in \cite{figallo2014algebraic}. They provided an algebraic axiomatization of IKt-algebras, showing its soundness and completeness within the class of IKt-algebras. The discrete duality was also shown in the same paper. It is notable that Simpson's intuitionistic modal logic \cite{SI1994} can be conservatively extended to Ewald's $\msf{IK}_t$. A sequent calculus of bi-intuitionistic tense logic has been showed in \cite{Gore2010}. Recently, some weakening of intuitionistic tense logic has been proposed and studied (cf. \cite{Liang2020OnTD,LL2019,PLL2021}). These weaknesses are beneficial attempts towards solving the open problem of decidability of intuitionistic tense logic. We follow this research line and present a new cut-elimination method about a weak variant of intuitionistic tense logic ($\msf{wIK}_t$ for short) in \cite{LL2019,PLL2021}.

    The present paper is organized as follows. Section \ref{section2} gives preliminaries about weak intuitionistic tense logic and its sequent calculus $\msf{GwIK}_{t}$. Section \ref{section3} shows the cut-elimination of $\msf{GwIK}_{t}$. Section \ref{section5} gives some concluding remarks.

\section{Preliminaries}\label{section2}

    In this section, we present some preliminaries related to the intuitionistic tense logic and its sequent calculus $\msf{GwIK}_{t}$. The corresponding soundness and completeness of $\msf{GwIK}_{t}$ with respect to the intuitionistic tense algebra or Heyting algebra with tense operators can be checked in \cite{PLL2021}.

    \begin{definition}
    The set of all formulas $\mathcal{F}$ is defined inductively as follows:
    \[
    \mathcal{F}\ni \alpha::= p\mid\top\mid\bot \mid(\alpha_1\imp\alpha_2)\mid (\alpha_1\wedge\alpha_2)\mid(\alpha_1\vee\alpha_2)\mid\D\al\mid\B\al\mid\bd\al\mid\bb\al
    \]
    where $p\in\mbf{Var}=\{p_i:i<\omega\}$, a denumerable set of variables. We use the abbreviation $\neg\alpha:=\al\imp\bot$. Formulas belong to $\mbf{Atom}=\mbf{Var}\cup\{\bot,\top\}$ are called {\em atomic}. The {\em complexity} of a formula $\al$, denoted by $c(\al)$, is defined as usual. Let $Sub(\al)$ be the set of all subformulas of $\al$. For a set of formulas $\Gamma$, let $Sub(\Gamma)=\bigcup_{\al\in\Ga} Sub(\al)$. A {\em substitution} is a homomorphism $\si:\mc{F}\to\mc{F}$.
    \end{definition}
    
    \begin{definition}
    The Hilbert-style axiomatic system $\msf{wIK}_t$ consists of the following axiom schemata and inference rules:
    
        \item[$(1)$] Axiom Schemata:
            \begin{quote}
            $\mrm{(IPC)}$: All axioms of intuitionistic propositional calculus. 
            
            $\mrm{(Dual_{\D\B})}$: $\B\neg\al\imp\neg\D\al$.
            
            $\mrm{(Dual_{\bd\bb})}$: $\bb\neg\al\imp\bd\neg\al$.
            \end{quote}
        \item[$(2)$] Inference Rules:
            \[
            \AxiomC{$\D\al\imp\be$}
            \doubleLine
            \RightLabel{$(\mrm{Adj}_{\D\bb})$}
            \UnaryInfC{$\al\imp\bb\be$}
            \DisplayProof
            \quad
            \AxiomC{$\bd\al\imp\be$}
            \doubleLine
            \RightLabel{$(\mrm{Adj}_{\bd\B})$}
            \UnaryInfC{$\al\imp\B\be$}
            \DisplayProof
            \quad
            \frac{\al\quad\al\imp\be}{\be}(\mrm{MP})
            \]
    \end{definition}

    \begin{definition}
    Let the comma, $\cc$ and $\bu$ be structural counterparts for $\land$, $\D$ and $\bd$ respectively. The set of all formula
    structures $\mathcal{FS}$ is defined inductively as follows:
    \[\mathcal{FS} \ni \Gamma ::= \alpha \mid (\Gamma_1,\Gamma_2)\mid \cc\Gamma\mid\bu\Ga
    \]
    Let $\mc{FS}^\ep=\mc{FS}\cup\{\epsilon\}$ where $\epsilon$ stands for the empty formula structure. By $f(\Gamma)$ we denote the formula obtained from $\Gamma$ by replacing all structure operators with their corresponding formula connectives. For every structure operators $\Ga\in\mc{FS}^\ep$, $f(\Ga)$ is defined inductively as follows:
    \[
    f(\ep)=\top\quad f(\al)=\al\quad f(\Ga_1,\Ga_2)=f(\Ga_1)\land f(\Ga_2)\quad f(\cc\Ga)=\D f(\Ga)\quad f(\bu\Ga)=\bd f(\Ga)
    \]
    Every nonempty formula structure $\Ga$ has a parsing tree with formulas in $\Ga$ as leaf nodes and structural operators in $\Ga$ as non-leaf nodes. A {\em sequent} is an expression of the form $\Gamma\Imp \be$ where $\Gamma\in\mc{FS}^\ep$ is a formula structure and $\be\in\mc{F}$ is a formula.
    \end{definition}

    \begin{definition}
    Let $-$ be the symbol called the {\em position}. A {\em context} $\Gamma[-]$ is a formula structure $\Gamma\in\mc{FS}^\ep$ together with a designated position $[-]$ which can be filled with a formula structure. The set of all contexts $\mc{C}$ is defined inductively as follows: 
    \[
    \mc{C}\ni\Gamma[-]::= - \mid (\Gamma_1[-],\Gamma_2) \mid (\Gamma_1,\Gamma_2[-]) \mid \circ\Gamma[-]\mid \bu\Gamma[-]
    \]
    Let $\Gamma[\Delta]$ be the formula structure obtained from context $\Gamma[-]\in\mc{C}$ by substituting $\Delta\in\mc{FS}$ for position $-$.
    \end{definition}
    
    \begin{definition}
    The sequent calculus $\msf{GwIK}_{t}$ for $\msf{wIK}_t$ consists of the following axiom schemata and sequent rules: for $i\in\{1,2\}$,
        \begin{quote}
            \item[$(1)$] Axiom schemata:
            \[
            (\mrm{Id})~\al\Imp\al
            \]
            \item[$(2)$] Logical rules:
            \[
            \frac{\Ga[\top]\seq\be}{\Ga[\De]\seq\be}{(\top)}
            \quad
            \frac{\De\seq\bot}{\Ga[\De]\seq\be}{(\bot)}
            \]
            \[
            \frac{\Gamma[\alpha_1,\alpha_2]\seq\beta}{\Gamma[\alpha_1\wedge\alpha_2]\seq\beta}{({\wedge}\mrm{L})}
            \quad
            \frac{\Gamma \seq\be_1\quad\Gamma\seq\be_2}{\Gamma\seq\be_1\wedge\be_2}{({\wedge}\mrm{R})}
            \]
            \[
            \frac{\Gamma[\al_1]\seq\beta\quad\Gamma[\al_2]\seq\beta}{\Gamma[\al_1\vee\al_2]\seq\be}{({\vee}\mrm{L})}
            \quad
            \frac{\Gamma\seq\beta_i}{\Gamma\seq\beta_1\vee\beta_2}{({\vee}\mrm{R})}
            \]
            \[
            \frac{\De\Imp\al_1\quad\Ga[\al_2]\Imp\be}{\Ga[\De,\al_1\imp\al_2]\Imp\be}{({\imp}\mrm{L})}
            \quad
            \frac{\be_1,\Ga\Imp\be_2}{\Ga\Imp\be_1\imp\be_2}{({\imp}\mrm{R})}
            \]
        \item[$(3)$] Tense rules:
            \[
            \frac{\Gamma[\circ \alpha]\seq\beta}{\Gamma[\D\alpha]\seq \beta}{(\D\mrm{L})}
            \quad
            \frac{ \Gamma\seq\be}{\circ \Gamma\seq\D\be}{(\D\mrm{R})}
            \quad
            \frac{\Gamma[\bu\alpha]\seq\beta}{\Gamma[\blacklozenge\alpha]\seq \beta}{(\bd\mrm{L})}
            \quad
            \frac{ \Gamma\seq\be}{\bu\Gamma\seq\blacklozenge\be}{(\bd\mrm{R})}
            \]
            \[
            \frac{\Gamma[\alpha]\seq \beta}{ \Gamma[\circ\bb \alpha]\seq \beta} {(\bb\mrm{L})}
            \quad
            \frac{\circ\Gamma\seq \be}{ \Gamma\seq \bb\be} {(\bb\mrm{R})}
            \quad
            \frac{\Gamma[\alpha]\seq \beta}{ \Gamma[\bu \B \alpha]\seq \beta} {(\B\mrm{L})}
            \quad
            \frac{\bu\Gamma\seq \be}{ \Gamma\seq \B\be} {(\B\mrm{R})}
            \]	
            \item[$(4)$] Structural rules:
            \[
            \frac{\Gamma[\circ\Delta_1,\circ\Delta_2]\seq \beta}{\Gamma[\circ(\Delta_1,\Delta_2)]\seq \beta}{\mrm{(\mrm{Con}_{\circ})}}
            \quad
            \frac{\Gamma[\bullet\Delta_1,\bu\Delta_2]\seq \beta}{\Gamma[\bullet(\Delta_1,\Delta_2)]\seq \beta}{\mrm{(\mrm{Con}_{\bullet})}}
            \]
            \[
            \frac{\Gamma[\alpha,\alpha]\seq \beta}{\Gamma[\alpha]\seq \beta}{(\mrm{Con_F})}
            \quad
            \frac{\Gamma[\Delta_i]\seq \beta}{\Gamma[\Delta_1,\Delta_2]\seq \beta}{\mrm{(Wk)}}
            \quad
            \frac{\Gamma[\Delta_1,\Delta_2]\seq \beta}{\Gamma[\Delta_2,\Delta_1]\seq \beta}{(\mrm{Ex})}
            \]
            \[
            \frac{\circ\Delta_1, \Delta_2\Imp \bot}{\Gamma[\Delta_1,\bu\Delta_2]\Imp \be}{(\mrm{Dual}_{\circ\bu})}
            \quad
            \frac{\bu\Delta_1,\Delta_2\Imp \bot}{\Gamma[\Delta_1, \cc\Delta_2]\Imp \be}{(\mrm{Dual}_{\bu\circ})}
            \]
        \item[$(5)$] Cut rule:
            \[
            \frac{\Delta\seq \alpha\quad \Gamma[\alpha]\seq \beta}{\Gamma[\Delta]\seq \beta}{(\mrm{Cut})}
            \]
        \end{quote}
    \end{definition}

    \begin{fact}\label{fact: admissible rules in GIKt}
        The following sequent rules are admissible in $\msf{GwIK}_{t}$: for $\flat\in\{\D,\B,\bd,\bb\}$,
        \[
        \frac{\al_1\Imp\be_1\quad\al_2\Imp\be_2}{\al_1\land\al_2\Imp\be_2\land\be_2}{(\land)}
        \quad
        \frac{\al_1\Imp\be_1\quad\al_2\Imp\be_2}{\al_1\vee\al_2\Imp\be_2\vee\be_2}{(\vee)}
        \]
        \[
        \frac{\al\Imp\al\quad\be\Imp\be}{\al\imp\be\Imp\al\imp\be}{(\imp)}
        \quad
        \frac{\Ga[\Delta,\Delta]\seq\beta}{\Ga[\Delta]\seq\beta}{(\mrm{Con})}
        \quad
        \frac{\al\Imp\be}{\flat\al\Imp\flat\be}{(\mrm{Mon})}
        \]
        \[
        \frac{\Gamma[\Delta_1,(\Delta_2,\Delta_3)]\seq \beta}{\Gamma[(\Delta_1,\Delta_2),\Delta_3]\seq \beta}{(\mrm{As_1})}
        \quad
        \frac{\Gamma[(\Delta_1,\Delta_2),\Delta_3]\seq \beta}{\Gamma[\Delta_1,(\Delta_2,\Delta_3)]\seq \beta}{(\mrm{As_2})}
        \]
        \end{fact}

\section{Cut-elimination}\label{section3}

    In this section, we show the cut-elimination result of $\msf{GwIK}_{t}$. We split some of the original rules to avoid some undesired situations when adopting the regular methods of cut-elimination to handle $\msf{GwIK}_{t}$. Firstly, we break down the original (Con$_\mrm{F}$) into several sub-rules and restrict the (Id) so that some connectives shall be inverse. Secondly, we split the (Cut) rule into several pieces as well, which corresponds to the last step's action. Finally, by preserving the derivability of these calculi, it suffices to show the elimination of all the pieces of (Cut) rules in the last calculus to show the cut-elimination of $\msf{GwIK}_{t}$. our proof strategy can be illustrated in Figure \ref{fig: proof strategy} as follows where the dashed line denotes the result we attempt to achieve:

    \begin{figure}[ht]
        \centering
        \begin{tikzpicture}
            \node(GIKt) at (-5,0) {$\msf{GwIK}_{t}$};
            \node(GIKtd)[right of=GIKt, xshift=2.8cm]{$\msf{GwIK}_{t}^\dagger$};
            \node(GIKtdd)[right of=GIKtd, xshift=2.8cm]{$\msf{GwIK}_{t}^\ddagger$};
            \draw [<->] (GIKt) -- (GIKtd) node[midway, above] {split the (Con$_\mrm{F}$)};
            \draw [<->] (GIKtdd) -- (GIKtd) node[midway, above] {split the (Cut)};
            \draw [<->, overlay, dashed] (GIKtdd) to[out=-10,in=-170] (GIKt);
        \end{tikzpicture}
        \caption{\centering{The proof strategy of $\msf{GwIK}_{t}$'s cut-elimination}}
        \label{fig: proof strategy}
    \end{figure}
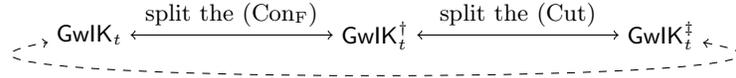

    The modified calculus $\msf{GwIK}_{t}^\dagger$ shall be obtained from $\msf{GwIK}_{t}$ as follows: the axiom schemata is restricted as (Id$_\mrm{A}$) $x\Imp x$ where $x\in\mbf{Atom}$. Further, the (Con$\mrm{_F}$) rule will be replaced by the rules below: for $x\in\mbf{Atom}$ and $\sharp\in\{\B,\bb\}$,
    \[
    \frac{\Gamma[x,x]\seq \beta}{\Gamma[x]\seq \beta}{(\mrm{Con_A})}
    \quad
    \frac{\Gamma[\sharp\al,\sharp\al]\seq \beta}{\Gamma[\sharp\al]\seq \beta}{(\mrm{Con_\sharp})}
    \quad
    \frac{\Gamma[\al_1\imp\al_2,\al_1\imp\al_2]\seq \beta}{\Gamma[\al_1\imp\al_2]\seq \beta}{(\mrm{Con_\imp})}
    \]

    \begin{lemma}\label{lem: Id admissible in GIKt^dagger}
    $\al\Imp\al$ is derivable in $\msf{GwIK}_{t}^\dagger$
    \end{lemma}

    \begin{proof}
    The proof proceeds by the induction on the complexity of $\al$. Clearly, the induction basis holds since one has (Id$_\mrm{A}$). The induction steps are guaranteed by ($\land$), ($\vee$), ($\imp$) and (Mon) in Fact \ref{fact: admissible rules in GIKt}.
    \end{proof}

    \begin{lemma}\label{lem: left rule inverse}
    For any $n\ge1$, the following hold in cut-free $\msf{GwIK}_{t}^\dagger$:
        \begin{quote}
            \begin{itemize}
                \item[$(1)$] If $\vd^n\Ga[\al_1\land\al_2]\Imp\be$, then $\vd^n\Ga[\al_1,\al_2]\Imp\be$.
                \item[$(2)$] If $\vd^n\Ga[\al_1\vee\al_2]\Imp\be$, then $\vd^n\Ga[\al_1]\Imp\be$ and $\vd^n\Ga[\al_2]\Imp\be$.
                \item[$(3)$] If $\vd^n\Ga[\D\al]\Imp\be$, then $\vd^n\Ga[\cc\al]\Imp\be$.
                \item[$(4)$] If $\vd^n\Ga[\bd\al]\Imp\be$, then $\vd^n\Ga[\bu\al]\Imp\be$.
            \end{itemize}
        \end{quote}
    \end{lemma}

    \begin{proof}
    The proof proceeds by the induction on the height of derivation. We only show the proof of (1), others can be treated similarly.  Assume $\vd^n\Ga[\al_1\land\al_2]\Imp\be$. Let $n=1$, if the upper sequent of $\vd^1\Ga[\al_1\land\al_2]\Imp\be$ is an instance of (Id), then it must be obtained by the (Wk) rule. Thus one has $\vd^1\Ga[\al_1,\al_2]\Imp p$ by (Wk) on $p\Imp p$. 
    Assume $n>1$, and let the sequent be obtained by an arbitrary rule (R). If $\al_1\land\al_2$ is principal in (R), then (R) is ($\land$L) or ($\top$) or ($\bot$). For ($\land$L), one has $\vd^{n}\Ga[\al_1,\al_2]\Imp \be$ simply by the premise. For other cases, one has $\vd^{n}\Ga[\al_1,\al_2]\Imp \be$ by (R) on its premise. If $\al_1\land\al_2$ is not principal in (R), then one has $\vd^{n}\Ga[\al_1,\al_2]\Imp \be$ by induction hypothesis on the premise(s) and an application of rule (R).
    \end{proof}

    \begin{lemma}\label{lem: con_F admissible}
    The following rule is cut-free admissible in $\msf{GwIK}_{t}^\dagger$:
        \[
        \frac{\Gamma[\alpha,\alpha]\seq \beta}{\Gamma[\alpha]\seq \beta}{(\mrm{Con_F})}
        \]
    \end{lemma}

    \begin{proof}
    The proof proceeds by the induction on the complexity of $\al$ i.e. $c(\al)$. Assume $c(\al)=0$, the claim holds by $(\mrm{Con_A})$. Assume $c(\al)>0$, let $\al=\al_1\land\al_2$ or $\al_1\vee\al_2$ or $\D\al'$ or $\bd\al'$. Take $\al=\al_1\land\al_2$ as an example. By Lemma \ref{lem: left rule inverse} (1), one has $\vd\Gamma[\al_1,\al_2,\al_1,\al_2]\seq \beta$. By (Ex) and induction hypothesis, one has $\vd\Gamma[\al_1,\al_2]\seq \beta$. By ($\land$L), one has $\vd\Gamma[\al_1\land\al_2]\seq \beta$. Let $\al=\al_1\imp\al_2$ or $\B\al'$ or $\bb\al'$. The claim is held by $(\mrm{Con_\imp})$ and $(\mrm{Con_\sharp})$.
    \end{proof}

    \begin{corollary}\label{coro: structure contraction admissible}
    The following rule is cut-free admissible in $\msf{GwIK}_{t}^\dagger$:
        \[
        \frac{\Gamma[\De,\De]\seq \beta}{\Gamma[\De]\seq \beta}{(\mrm{Con})}
        \]
    \end{corollary}

    \begin{proof}
    It follows from Lemma \ref{lem: con_F admissible}, $(\mrm{Con}_{\circ})$ and $(\mrm{Con}_{\bu})$.
    \end{proof}

    \begin{lemma}\label{lem: GIKt iff GIKt^dagger}
    For any sequent $\Ga[\De]\Imp\be$,  $\msf{GwIK}_{t}\vd\Ga[\De]\Imp\be$ iff $\msf{GwIK}_{t}^\dagger\vd\Ga[\De]\Imp\be$.
    \end{lemma}

    \begin{proof}
    For the if part,  by Lemma \ref{lem: Id admissible in GIKt^dagger} and \ref{lem: con_F admissible}, (Id) and $\mrm{(Con_F)}$ are admissible in $\msf{GwIK}_{t}^\dagger$. Thus all the axioms and rules are preserved in $\msf{GwIK}_{t}^\dagger$. For the only if part, observe that the axioms and rules replaced by $\msf{GwIK}_{t}^\dagger$ from $\msf{GwIK}_{t}$ are special cases of the original version. (Id$_\mrm{A}$) is a restricted case of (Id) and new contraction rules $(\mrm{Con_A})$, $(\mrm{Con_\sharp})$ and $(\mrm{Con_\imp})$ are part of the $\mrm{(Con_F)}$ rule. Clearly, any sequent provable in $\msf{GwIK}_{t}^\dagger$ is provable in $\msf{GwIK}_{t}$ since all the other rules remain the same.
    \end{proof}

    \begin{lemma}\label{lem: right rule inverse}
    For any $n\ge1$, the following hold in cut-free $\msf{GwIK}_{t}^\dagger$:
        \begin{quote}
            \begin{itemize}
                \item[$(1)$] If $\vd^n\Ga\Imp\be_1\imp\be_2$, then $\vd^n\be_1,\Ga\Imp\be_2$.
                \item[$(2)$] If $\vd^n\Ga\Imp\B\be$, then $\vd^n\bu\Ga\Imp\be$.
                \item[$(3)$] If $\vd^n\Ga\Imp\bb\be$, then $\vd^n\cc\Ga\Imp\be$.
            \end{itemize}
        \end{quote}
    \end{lemma}

    \begin{proof}
    The proof proceeds by the induction of the height of the derivation. Take (2) as an example, others can be treated similarly. Assume $n=1$, then $\vd^1\Ga\Imp\B\be$ must be obtained by ($\bot$) on $\bot\Imp\bot$. Thus $\vd^1\Ga[\bot]\Imp\B\be$. Clearly, one has $\vd^1\bu\Ga[\bot]\Imp\be$ by ($\bot$) on $\bot\Imp\bot$. Assume $n>1$ and let the sequent be obtained by an arbitrary rule (R). If $\B\be$ is principal in (R), then (R) is ($\B$R) or ($\bot$). For ($\B$R), one has $\vd^{n}\bu\Ga\Imp\be$ simply by the premise. For ($\bot$), the claim is held by premise and ($\bot$). If $\B\be$ is not principal in (R), then the claim is held by induction hypothesis on the premise and an application of rule (R).
    \end{proof}

    Further modifications on the (Cut) rule are important before we start cut-elimination. We split the entire (Cut) rule into 4 closely related rules as follows: for $x\in\mbf{Atom}$ and $\sharp\in\{\B,\bb\}$,
    \[
    \frac{\De\Imp x\quad\Ga[x]^n\Imp\be}{\Ga[\De]^n\Imp\be}{\mrm{(Mix_A)}}
    \quad
    \frac{\De\Imp\sharp\al\quad\Ga[\sharp\al]^n\Imp\be}{\Ga[\De]^n\Imp\be}{\mrm{(Mix_\sharp)}}
    \quad
    \]
    \[
    \frac{\De\Imp\al_1\imp\al_2\quad\Ga[\al_1\imp\al_2]^n\Imp\be}{\Ga[\De]^n\Imp\be}{\mrm{(Mix_\imp)}}
    \quad
    \frac{\De\Imp\al^\ast\quad\Ga[\al^\ast]\Imp\be}{\Ga[\De]\Imp\be}{\mrm{(Cut_\ast)}}
    \]
    where the cut formula $\al^\ast$ in $\mrm{(Cut_\ast)}$ cannot be the form of cut formulas in other modified cut rules i.e. $\al^\ast\not=x\mid\sharp\al\mid\al_1\imp\al_2$ where $x\in\mbf{Atom}$. We denote these four rules {\em (Cut)-like rules} and let $\msf{GwIK}_{t}^\ddagger$ denote the calculus that is obtained from $\msf{GwIK}_{t}^\dagger$ by replacing the (Cut) rule by the above (Cut)-like rules.

    \begin{lemma}\label{lem: GIKt^dagger iff GIKt^ddagger}
    For any sequent $\Ga[\De]\Imp\be$,  $\msf{GwIK}_{t}^\dagger\vd\Ga[\De]\Imp\be$ iff $\msf{GwIK}_{t}^\ddagger\vd\Ga[\De]\Imp\be$.
    \end{lemma}

    \begin{proof}
    The proof is similar to the Lemma \ref{lem: GIKt iff GIKt^dagger}.
    \end{proof}

    Therefore, by Lemma \ref{lem: GIKt^dagger iff GIKt^ddagger}, it suffices to show the cut-elimination of $\msf{GwIK}_{t}^\ddagger$ for the cut-elimination of $\msf{GwIK}_{t}$. Note that any derivation tree in $\msf{GwIK}_{t}^\ddagger$ may contain various applications of $\mrm{(Mix_A)}$, $\mrm{(Mix_\sharp)}$, $\mrm{(Mix_\imp)}$ and $\mrm{(Cut_\ast)}$. Thus one needs to eliminate all the application of these (Cut)-like rules to obtain the cut-elimination result.

    \begin{lemma}\label{lem: bot or top (Cut_A) elimination}
    If $x$ is $\bot$ or $\top$ in applications of $\mrm{(Mix_A)}$ in any derivation, then such a derivation can be rewritten without using $\mrm{(Mix_A)}$ i.e. $\mrm{(Mix_A)}$ can be eliminated when $x$ is $\bot$ or $\top$.
    \end{lemma}

    \begin{proof}
    For any sequent $\Ga[\De]\Imp\be$, assume $\msf{GwIK}_{t}^\ddagger\vd\Ga[\De]\Imp\be$ is obtained by $\mrm{(Mix_A)}$ when $x$ is $\bot$ or $\top$. Then one has the following derivations:
    \[
    \frac{\vd\De\Imp\bot\quad\vd\Ga[\bot]^n\Imp\be}{\vd\Ga[\De]^n\Imp\be}{\mrm{(Mix_A)}}
    \quad
    \frac{\vd\De\Imp\top\quad\vd\Ga[\top]^n\Imp\be}{\vd\Ga[\De]^n\Imp\be}{\mrm{(Mix_A)}}
    \]
    Clearly, by adopting ($\bot$) and ($\top$), these derivations can be transformed as follows:
    \[
    \frac{\vd\De\Imp\bot}{\vd\Ga[\De]^n\Imp\be}{\mrm{(\bot)}}
    \quad
    \frac{\vd\Ga[\top]^n\Imp\be}{\vd\Ga[\De]^n\Imp\be}{\mrm{(\top)}^n}
    \]
    Thus the claim holds.
    \end{proof}

    \begin{theorem}[Cut-elimination]\label{thm: GIKt^ddagger CE}
    If $\msf{GwIK}_{t}^\ddagger\vd\Ga[\De]\Imp\be$, then $\Ga[\De]\Imp\be$ is derivable in $\msf{GwIK}_{t}^\ddagger$ without any applications of  $\mrm{(Mix_A)}$, $\mrm{(Mix_\sharp)}$, $\mrm{(Mix_\imp)}$ and $\mrm{(Cut_\ast)}$.
    \end{theorem}

    \begin{proof}
    Assume $\msf{GwIK}_{t}^\ddagger\vd\Ga[\De]\Imp\be$, then there is a derivation $\mc{D}$ of it. We take one of the uppermost applications of (Cut)-like rules in this derivation. Certainly, this uppermost (Cut)-like rule could be one of the $\mrm{(Mix_A)}$, $\mrm{(Mix_\sharp)}$, $\mrm{(Mix_\imp)}$ and $\mrm{(Cut_\ast)}$. Then we shall eliminate all these (Cut)-like rules. 
    
    However, when eliminating the (Cut)-like rules by decreasing the complexity of the cut formula, we may encounter some (Cut)-like rules different from the one we attempt to eliminate. Let's denote these uncertain (Cut)-like rules (Cut$_?$). Nonetheless, (Cut$_?$) is nothing new but one of the (Cut)-like rules we shall deal with. For instance, for a cut formula $\sharp\al$, we may find out that $\al=\sharp\al'$ for some $\al'$ when we decrease the complexity of the cut formula. But $\al$ could be other forms as well such as $x\in\mbf{Atom}$ or $\al_1\odot\al_2$ where $\odot\in\{\land,\vee,\imp\}$ or $\natural\al'$ where $\natural\in\{\D,\bd\}$. We are supposed to eliminate these (Cut)-like rules one by one. Thus, it suffices to consider the following Scenarios (I-IV):
    
    Scenario (I): Suppose the uppermost (Cut)-like rule is $\mrm{(Mix_A)}$. By Lemma \ref{lem: bot or top (Cut_A) elimination}, the case that $x$ is $\bot$ or $\top$ has been settled. Let $x=p$ and suppose the derivation is as follows:
    \[
    \frac{\vd^{h_1}\De\Imp p\quad\vd^{h_2}\Ga[p]^n\Imp\be}{\vd\Ga[\De]^n\Imp\be}{\mrm{(Mix_A)}}
    \]
    The proof proceeds by the induction on $h_2$. Assume $h_2=0$, then $\vd\Ga[p]^n\Imp\be$ is an instance of (Id$_A$) and $\Ga[p]^n=\be=p$. Thus the conclusion is exactly the other upper sequent. Assume $h_2>0$, let $\vd\Ga[p]^n\Imp\be$ be obtained by an arbitrary rule (R). One has the following cases:
    
    (1) Suppose $p$ is not principal in (R), then one applies the $\mrm{(Mix_A)}$ to its premise and $\vd\De\Imp p$. Then by (R) again, the height of $\mrm{(Mix_A)}$ is decreased.

    (2) Suppose $p$ is principal in (R). Consider the following subcases:

    (2.1) (R) is ($\top$). Let $n_1+n_2=n$ and $\Ga[\De]^n=\Ga[\De]^{n_1}[\Si[\De]^{n_2}]$. Suppose the derivation ends with:
        \begin{prooftree}
            \AxiomC{$\De\Imp p$}
            \AxiomC{$\Ga[p]^{n_1}[\top]\Imp\be$}
            \RightLabel{$\mrm{(\top)}$}
            \UnaryInfC{$\Ga[p]^{n_1}[\Si[p]^{n_2}]\Imp\be$}
            \RightLabel{$\mrm{(Mix_A)}$}
            \BinaryInfC{$\Ga[\De]^{n_1}[\Si[\De]^{n_2}]\Imp\be$}
        \end{prooftree}
    Then the derivation can be transformed into: 
        \begin{prooftree}
            \AxiomC{$\De\Imp p$}
            \AxiomC{$\Ga[p]^{n_1}[\top]\Imp\be$}
            \RightLabel{$\mrm{(Mix_A)}$}
            \BinaryInfC{$\Ga'[\De]^{n_1}[\top]\Imp\be$}
            \RightLabel{$\mrm{(\top)}$}
            \UnaryInfC{$\Ga[\De]^{n_1}[\Si[\De]^{n_2}]\Imp\be$}
        \end{prooftree}

    (2.2) (R) is ($\bot$). Let $n_1+n_2=n$ and $\Ga[\De]^n=\Ga[\De]^{n_1}[\Si[\De]^{n_2}]$. Suppose the derivation ends with:
        \begin{prooftree}
            \AxiomC{$\De\Imp p$}
            \AxiomC{$\Si[p]^{n_2}\Imp\bot$}
            \RightLabel{$\mrm{(\bot)}$}
            \UnaryInfC{$\Ga[p]^{n_1}[\Si[p]^{n_2}]\Imp\be$}
            \RightLabel{$\mrm{(Mix_A)}$}
            \BinaryInfC{$\Ga[\De]^{n_1}[\Si[\De]^{n_2}]\Imp\be$}
        \end{prooftree}
    Then the derivation can be transformed into: 
        \begin{prooftree}
            \AxiomC{$\De\Imp p$}
            \AxiomC{$\Si[p]^{n_2}\Imp\bot$}
            \RightLabel{$\mrm{(Mix_A)}$}
            \BinaryInfC{$\Si[\De]^{n_2}\Imp\bot$}
            \RightLabel{$\mrm{(\bot)}$}
            \UnaryInfC{$\Ga[\De]^{n_1}[\Si[\De]^{n_2}]\Imp\be$}
        \end{prooftree}

    (2.3) (R) is ($\mrm{Con_A}$). Suppose the derivation ends with:
        \begin{prooftree}
            \AxiomC{$\De\Imp p$}
            \AxiomC{$\Ga[p,p][p]^{n-1}\Imp\be$}
            \RightLabel{$\mrm{(Con_A)}$}
            \UnaryInfC{$\Ga[p][p]^{n-1}\Imp\be$}
            \RightLabel{$\mrm{(Mix_A)}$}
            \BinaryInfC{$\Ga[\De][\De]^{n-1}\Imp\be$}
        \end{prooftree}
    Then the derivation can be transformed into: 
        \begin{prooftree}
            \AxiomC{$\De\Imp p$}
            \AxiomC{$\Ga[p,p][p]^{n-1}\Imp\be$}
            \RightLabel{$\mrm{(Mix_A)}$}
            \BinaryInfC{$\Ga[\De,\De][\De]^{n-1}\Imp\be$}
            \RightLabel{$\mrm{(Con)}$ by Corollary \ref{coro: structure contraction admissible}}
            \UnaryInfC{$\Ga[\De][\De]^{n-1}\Imp\be$}
        \end{prooftree}

    (2.4) (R) is ($\mrm{Wk}$). Let $n_1+n_2=n$ and $\Ga[\De]^n=\Ga_1[\De]^{n_1},\Ga_2[\De]^{n_2}$. Suppose the derivation ends with: for $i\in\{1,2\}$,
        \begin{prooftree}
            \AxiomC{$\De\Imp p$}
            \AxiomC{$\Ga_i[p]^{n_1}\Imp\be$}
            \RightLabel{$\mrm{(Wk)}$}
            \UnaryInfC{$\Ga_1[p]^{n_1},\Ga_2[p]^{n_2}\Imp\be$}
            \RightLabel{$\mrm{(Mix_A)}$}
            \BinaryInfC{$\Ga_1[\De]^{n_1},\Ga_2[\De]^{n_2}\Imp\be$}
        \end{prooftree}
    Then the derivation can be transformed into: 
        \begin{prooftree}
            \AxiomC{$\De\Imp p$}
            \AxiomC{$\Ga_i[p]^{n_1}\Imp\be$}
            \RightLabel{$\mrm{(Mix_A)}$}
            \BinaryInfC{$\Ga_i[\De]^{n_1}\Imp\be$}
            \RightLabel{$\mrm{(Wk)}$}
            \UnaryInfC{$\Ga_1[\De]^{n_1},\Ga_2[\De]^{n_2}\Imp\be$} 
        \end{prooftree}

    Scenario (II): Suppose the uppermost (Cut)-like rule is $\mrm{(Mix_\sharp)}$. Suppose the derivation is as follows: for $\sharp\in\{\B,\bb\}$,
    \[
    \frac{\vd^{h_1}\De\Imp\sharp\al\quad\vd^{h_2}\Ga[\sharp\al]^n\Imp\be}{\vd\Ga[\De]^n\Imp\be}{\mrm{(Mix_\sharp)}}
    \]
    The proof proceeds by the induction on $h_2\ge1$ and the complexity of the cut formula i.e. $c(\sharp\al)$. Assume $h_2=1$, then $\vd\Ga[\sharp\al]^n\Imp\be$ is obtained by ($\sharp$L) or ($\top$) or ($\bot$) or (Wk). We take the premise of $\vd\Ga[\sharp\al]^n\Imp\be$ to be $\vd\top\Imp\top$ as an example, others can be treated similarly. For ($\sharp$L), we take ($\B$L) as an example, and then $\Ga[\sharp\al]^n\Imp\be$ is $\bu\B\top\Imp\top$. Thus one has $\vd\bu\De\Imp\top$ by Lemma \ref{lem: right rule inverse} (2) on the other upper sequent. For ($\top$), one has $\Ga[\sharp\al]^n\Imp\be$ is $\Si[\sharp\al]^n\Imp\top$. Thus one has $\vd\Si[\De]^n\Imp\top$ by ($\top$) on $\vd\top\Imp\top$. For (Wk), one has $\Ga[\sharp\al]^n\Imp\be$ is $\Si[\sharp\al]^n,\top\Imp\top$. Thus one has $\vd\Si[\De]^n,\top\Imp\top$ by (Wk) on $\vd\top\Imp\top$. Assume $h_2>1$, let $\vd\Ga[\sharp\al]^n\Imp\be$ be obtained by an arbitrary rule (R). One has the following cases:

    (1) Suppose $\sharp\al$ is not principal in (R), then one applies the $\mrm{(Mix_\sharp)}$ to its premise and $\vd\De\Imp\sharp\al$. Then by (R) again, the height of $\mrm{(Mix_\sharp)}$ is decreased.

    (2) Suppose $\sharp\al$ is principal in (R). Consider the following subcases:

    (2.1) (R) is ($\top$) or ($\bot$) or (Con$_\sharp$) or (Wk). The proofs are quite similar to the Case (2) of Scenario (I).

    (2.2) (R) is ($\sharp$L). Take ($\B$L) as an example. Let $\Ga[\De]^n=\Si[\bu\De][\De]^{n-1}$. Suppose the derivation ends with:
        \begin{prooftree}
            \AxiomC{$\De\Imp\B\al$}
            \AxiomC{$\Si[\al][\B\al]^{n-1}$}
            \RightLabel{$\mrm{(\B L)}$}
            \UnaryInfC{$\Si[\bu\B\al][\B\al]^{n-1}$}
            \RightLabel{$\mrm{(Mix_\sharp)}$}
            \BinaryInfC{$\Si[\bu\De][\De]^{n-1}$}
        \end{prooftree}
    By Lemma \ref{lem: right rule inverse} (2) on $\vd\De\Imp\B\al$, one has $\vd\bu\De\Imp\al$. Then the derivation can be transformed into: 
        \begin{prooftree}
            \AxiomC{$\bu\De\Imp\al$}
            \AxiomC{$\De\Imp\B\al$}
            \AxiomC{$\Si[\al][\B\al]^{n-1}$}
            \RightLabel{$\mrm{(Mix_\sharp)}$}
            \BinaryInfC{$\Si[\al][\De]^{n-1}$}
            \RightLabel{$\mrm{(Cut_?)}$}
            \BinaryInfC{$\Si[\bu\De][\De]^{n-1}$}
        \end{prooftree}
    Note that the height of $\mrm{(Mix_\sharp)}$ and the complexity of the cut formula in $\mrm{(Cut_?)}$ has been decreased.

    Scenario (III): Suppose the uppermost (Cut)-like rule is $\mrm{(Mix_\imp)}$. Suppose the derivation is as follows:
    \[
    \frac{\vd^{h_1}\De\Imp\al_1\imp\al_2\quad\vd^{h_2}\Ga[\al_1\imp\al_2]^n\Imp\be}{\vd\Ga[\De]^n\Imp\be}{\mrm{(Mix_\imp)}}
    \]
    The proof proceeds by the induction on $h_2\ge1$ and the complexity of the cut formula i.e. $c(\al_1\imp\al_2)$. Assume $h_2=1$, then $\vd\Ga[\al_1\imp\al_2]^n\Imp\be$ is obtained by ($\imp$L) or ($\top$) or ($\bot$) or (Wk). Take ($\imp$L) as an example. Then $\Ga[\al_1\imp\al_2]^n\Imp\be$ is $p,p\imp q\Imp q$. Thus one has $p,\De\Imp q$ by Lemma \ref{lem: right rule inverse} (1) on the other upper sequent. Other cases can be treated similarly to those of Scenario (II). Assume $h_2>1$, let $\vd\Ga[\al_1\imp\al_2]^n\Imp\be$ be obtained by an arbitrary rule (R). One has the following cases:

    (1) Suppose $\al_1\imp\al_2$ is not principal in (R), then one applies the $\mrm{(Mix_\imp)}$ to its premise and $\vd\De\Imp\al_1\imp\al_2$. Then by (R) again, the height of $\mrm{(Mix_\imp)}$ is decreased.

    (2) Suppose $\al_1\imp\al_2$ is principal in (R). Consider the following subcases:

    (2.1) (R) is ($\top$) or ($\bot$) or (Con$_\imp$) or (Wk). The proofs are quite similar to the Case (2) of Scenario (I).

    (2.2) (R) is ($\imp$L). Let $n_1+n_2+1=n$ and $\Ga[\De]^n=\Ga[\De]^{n_1}[\Si[\De]^{n_2},\De]$. Suppose the derivation ends with:
        \begin{prooftree}
            \AxiomC{$\De\Imp\al_1\imp\al_2$}
            \AxiomC{$\Si[\al_1\imp\al_2]^{n_2}\Imp\al_1$}
            \AxiomC{$\Ga[\al_1\imp\al_2]^{n_1}[\al_2]\Imp\be$}
            \RightLabel{$\mrm{(\imp L)}$}
            \BinaryInfC{$\Ga[\al_1\imp\al_2]^{n_1}[\Si[\al_1\imp\al_2]^{n_2},\al_1\imp\al_2]\Imp\be$}
            \RightLabel{$\mrm{(Mix_\imp)}$}
            \BinaryInfC{$\Ga[\De]^{n_1}[\Si[\De]^{n_2},\De]\Imp\be$}
        \end{prooftree}
    By Lemma \ref{lem: right rule inverse} (1) on $\vd\De\Imp\al_1\imp\al_2$, one has $\vd\al_1,\De\Imp\al_2$. Then the derivation can be transformed into the following two parts: 
        \begin{prooftree}
            \AxiomC{$\De\Imp\al_1\imp\al_2$}
            \AxiomC{$\Si[\al_1\imp\al_2]^{n_2}\Imp\al_1$}
            \RightLabel{$\mrm{(Mix_\imp)}$}
            \BinaryInfC{$\Si[\De]^{n_2}\Imp\al_1$}
            \AxiomC{$\al_1,\De\Imp\al_2$}
            \RightLabel{$\mrm{(Cut_?)}$}
            \BinaryInfC{$\Si[\De]^{n_2},\De\Imp\al_2$}
        \end{prooftree}

        \begin{tikzpicture}[overlay, remember picture]
            \draw[->, overlay] (4.8,0.4) to[out=180,in=120] (1.5,-1);
        \end{tikzpicture}

        \begin{prooftree}
            \AxiomC{$\Si[\De]^{n_2},\De\Imp\al_2$}
            \AxiomC{$\De\Imp\al_1\imp\al_2$}
            \AxiomC{$\Ga[\al_1\imp\al_2]^{n_1}[\al_2]\Imp\be$}
            \RightLabel{$\mrm{(Mix_\imp)}$}
            \BinaryInfC{$\Ga[\De]^{n_1}[\al_2]\Imp\be$}
            \RightLabel{$\mrm{(Cut_?)}$}
            \BinaryInfC{$\Ga[\De]^{n_1}[\Si[\De]^{n_2},\De]\Imp\be$}
        \end{prooftree}
    Note that the height of $\mrm{(Mix_\imp)}$ and the complexity of the cut formulas in $\mrm{(Cut_?)}$ have been decreased.

    Scenario (IV): Suppose the uppermost (Cut)-like rule is $\mrm{(Cut_\ast)}$. For $\al^\ast\not=x\mid\sharp\al\mid\al_1\imp\al_2$ where $x\in\mbf{Atom}$ and $\sharp\in\{\B,\bb\}$, suppose the derivation is as follows: 
    \[
    \frac{\vd^{h_1}\De\Imp\al^\ast\quad\vd^{h_2}\Ga[\al^\ast]\Imp\be}{\vd\Ga[\De]\Imp\be}{\mrm{(Cut_\ast)}}
    \]
    The proof proceeds by the induction on $h_1\ge1$ and the complexity of the cut formula i.e. $c(\al^\ast)$. Assume $h_1=1$, then $\vd\De\Imp\al^\ast$ could be obtained from $\vd x\Imp x$ only by ($\land$R), ($\vee$R), ($\natural$R) where $x\in\mbf{Atom}$ and $\natural\in\{\D,\bd\}$. In addition, it can be obtained by ($\bot$) on $\vd\bot\Imp\bot$ as well. Thus we take $\vd\bot\Imp\bot$ as its premise for an example. Let the rule be ($\bot$), then the conclusion of $\mrm{(Cut_\ast)}$ is $\Ga[\De[\bot]]\Imp\be$ where $\De=\De[\bot]$. One can obtain same conclusion just by ($\bot$) on $\vd\bot\Imp\bot$ without using $\mrm{(Cut_\ast)}$. Let the rule be ($\land$L), then the right premise of $\mrm{(Cut_\ast)}$ is $\vd\Ga[\bot\land \bot]\Imp\be$ where $\al^\ast=\bot\land \bot$. By Lemma \ref{lem: left rule inverse} (1), one has $\vd\Ga[\bot,\bot]\Imp\be$. Thus by (Con$_\mrm{A}$), one has $\vd\Ga[\bot]\Imp\be$, which is just the conclusion of $\mrm{(Cut_\ast)}$ since the other upper sequent is $\vd \bot\Imp \bot\land \bot$. Let the rule be ($\vee$R) or ($\natural$R), we take ($\D$L) as an example. The right premise of $\mrm{(Cut_\ast)}$ is $\vd\Ga[\D\bot]\Imp\be$ where $\al^\ast=\D\bot$. By Lemma \ref{lem: left rule inverse} (3), one has $\vd\Ga[\cc\bot]\Imp\be$, which is exactly the conclusion of the $\mrm{(Cut_\ast)}$. Assume $h_1>1$, let $\vd^{h_1}\De\Imp\al^\ast$ be obtained by an arbitrary rule (R). One has the following cases:

    (1) Suppose $\al^\ast$ is not principal in (R), then one applies the $\mrm{(Cut_\ast)}$ to its premise and $\vd\Ga[\al^\ast]\Imp\be$. Then by (R) again, the height of $\mrm{(Cut_\ast)}$ is decreased.

    (2) Suppose $\al^\ast$ is principal in (R). Consider the following subcases:

    (2.1) (R) is ($\bot$). Let $\Ga[\De]=\Ga[\De[\Si]]$. Suppose the derivation ends with:
        \begin{prooftree}
            \AxiomC{$\Si\Imp\bot$}
            \RightLabel{$\mrm{(\bot)}$}
            \UnaryInfC{$\De[\Si]\Imp\al^\ast$}
            \AxiomC{$\Ga[\al^\ast]\Imp\be$}
            \RightLabel{$\mrm{(Cut_\ast)}$}
            \BinaryInfC{$\Ga[\De[\Si]]\Imp\be$}
        \end{prooftree}
    Then the derivation can be transformed into:
        \begin{prooftree}
            \AxiomC{$\Si\Imp\bot$}
            \RightLabel{$\mrm{(\bot)}$}
            \UnaryInfC{$\Ga[\De[\Si]]\Imp\be$}
        \end{prooftree}

    (2.2) (R) is ($\land$R). Suppose the derivation ends with:
        \begin{prooftree}
            \AxiomC{$\De\Imp\al_1$}
            \AxiomC{$\De\Imp\al_2$}
            \RightLabel{$\mrm{(\land R)}$}
            \BinaryInfC{$\De\Imp\al_1\land\al_2$}
            \AxiomC{$\Ga[\al_1\land\al_2]\Imp\be$}
            \RightLabel{$\mrm{(Cut_\ast)}$}
            \BinaryInfC{$\Ga[\De]\Imp\be$}
        \end{prooftree}
    By Lemma \ref{lem: left rule inverse} (1) on $\vd\Ga[\al_1\land\al_2]\Imp\be$, one has $\vd\Ga[\al_1,\al_2]\Imp\be$. Then the derivation can be transformed into:
        \begin{prooftree}
            \AxiomC{$\De\Imp\al_2$}
            \AxiomC{$\De\Imp\al_1$}
            \AxiomC{$\Ga[\al_1,\al_2]\Imp\be$}
            \RightLabel{$\mrm{(Cut_?)}$}
            \BinaryInfC{$\Ga[\De,\al_2]\Imp\be$}
            \RightLabel{$\mrm{(Cut_?)}$}
            \BinaryInfC{$\Ga[\De,\De]\Imp\be$}
            \RightLabel{$\mrm{(Con)}$ by Corollary \ref{coro: structure contraction admissible}}
            \UnaryInfC{$\Ga[\De]\Imp\be$}
        \end{prooftree}
    Note that the complexity of the cut formulas in $\mrm{(Cut_?)}$ has been decreased.

    (2.3) (R) is ($\vee$R) or ($\natural$L) where $\natural\in\{\D,\bd\}$. We take ($\vee$R) as an example, others can be treated similarly. Suppose the derivation ends with: for $i\in\{1,2\}$,
        \begin{prooftree}
            \AxiomC{$\De\Imp\al_i$}
            \RightLabel{$\mrm{(\vee R)}$}
            \UnaryInfC{$\De\Imp\al_1\vee\al_2$}
            \AxiomC{$\Ga[\al_1\vee\al_2]\Imp\be$}
            \RightLabel{$\mrm{(Cut_\ast)}$}
            \BinaryInfC{$\Ga[\De]\Imp\be$}
        \end{prooftree}
    By Lemma \ref{lem: left rule inverse} (2) on $\vd\Ga[\al_1\vee\al_2]\Imp\be$, one has $\vd\Ga[\al_i]\Imp\be$. Then the derivation can be transformed into:
        \begin{prooftree}
            \AxiomC{$\De\Imp\al_i$}
            \AxiomC{$\Ga[\al_i]\Imp\be$}
            \RightLabel{$\mrm{(Cut_?)}$}
            \BinaryInfC{$\Ga[\De]\Imp\be$}
        \end{prooftree}
    Note that the complexity of the cut formulas in $\mrm{(Cut_?)}$ has been decreased.
    \end{proof}

    \begin{theorem}[Cut-elimination]
    $\msf{GwIK}t$ has cut-elimination result.
    \end{theorem}

    \begin{proof}
    It follows from Lemma \ref{lem: GIKt iff GIKt^dagger}, Lemma \ref{lem: GIKt^dagger iff GIKt^ddagger} and Theorem \ref{thm: GIKt^ddagger CE}.
    \end{proof}

\section{Concluding Remarks}\label{section5}

    The key parts of adopting this method to prove cut-elimination are combinations of rules' invertibility and their corresponding mix rules. The process of cut-elimination can be simplified if the requirement of height-preserving invertibility is removed. Clearly, this method can be applied to various other logics and a general method shall be developed.

%
%
%
\bibliographystyle{splncs04}
\bibliography{reference.bib}

\begin{thebibliography}{1}
\providecommand{\url}[1]{\texttt{#1}}
\providecommand{\urlprefix}{URL }
\providecommand{\doi}[1]{https://doi.org/#1}

\bibitem{EW1986}
Ewald, W.B.: Intuitionistic tense and modal logic. The Journal of Symbolic
  Logic  \textbf{51}(1),  166--179 (1986)

\bibitem{figallo2014algebraic}
Figallo, A.V., Pelaitay, G.: An algebraic axiomatization of the {E}wald’s
  intuitionistic tense logic. Soft Computing  \textbf{18}(10),  1873--1883
  (2014)

\bibitem{Gore2010}
Gor\'e, R., Postniece, L., Tiu, A.: Cut-elimination and proof search for
  bi-intuitionistic tense logic. In: Goranko, V., Shehtman, V. (eds.) Advances
  in Modal Logic, vol.~8, pp. 156--177. CSLI Publications (2010)

\bibitem{Liang2020OnTD}
Liang, F., Lin, Z.: On the decidability of intuitionistic tense logic without
  disjunction. In: Proceedings of the Twenty-Ninth International Conference on
  International Joint Conferences on Artificial Intelligence. pp. 1798--1804
  (2021)

\bibitem{LL2019}
Lin, K., Lin, Z.: The sequent systems and algebraic semantics of intuitionistic
  tense logics. In: International Workshop on Logic, Rationality and
  Interaction. Lecture Notes in Computer Science, vol. 11813, pp. 140--152
  (2019)

\bibitem{PLL2021}
Peng, Y., Lin, Z., Liang, F.: On the finite model property of weak
  intuitionistic tense logic. In: Logic, Rationality, and Interaction: 8th
  International Workshop. Lecture Notes in Computer Science, vol. 13039, pp.
  174--182 (2021)

\bibitem{SI1994}
Simpson, A.K.: The proof theory and semantics of intuitionistic modal logic.
  Ph.D. thesis, University of Edinburgh, UK (1994)

\end{thebibliography}

\end{document}